\documentclass{ro}
\usepackage[colorlinks=true,citecolor=blue]{hyperref}
\usepackage{mathptmx, amsmath, amssymb, amsfonts, amsthm, mathptmx, enumerate, color,mathrsfs}
\usepackage{graphicx}
\usepackage{float}

\usepackage{epstopdf}

\newtheorem{theorem}{Theorem}[section]
\newtheorem{lemma}[theorem]{Lemma}

\theoremstyle{definition}

\numberwithin{equation}{section}
\newcommand{\R}{\mathbb{R}}

\newcommand{\N}{\mathbb{N}}

\newcommand{\I}{\mathbb{I}}

\newcommand{\smallO}[1]{\ensuremath{\mathop{}\mathopen{}o\mathopen{}\left(#1\right)}}

\begin{document}

\title{A new approach for solving  the linear  complementarity problem using smoothing functions}

\author{El hassene Osmani}\sameaddress{,2, *}\address{Laboratory of Fundamental and Numerical Mathematics, University Ferhat Abbas of Setif 1, Setif, Algeria.}
\author{ Mounir Haddou }\address{INSA Rennes, CNRS, IRMAR-UMR 6625, University of Rennes, Rennes, France.}
\author{Lina Abdallah}\address{Lebanese University, Tripoli, Lebanon.}
\subjclass{47H05, 90C33.}
\date{\emph{Keywords.} Linear complementarity problem, Newton's method, smoothing functions, $  \mathcal{P}$-matrix, interior-point methods, Soft-Max.}
\author{Naceurdine Bensalem}\sameaddress{1}

\begin{abstract} 
	
	Based on smoothing techniques, we propose two new methods to solve linear complementarity problems (LCP) called TLCP and Soft-Max. The idea of these  two new methods takes inspiration from interior-point methods in optimization. The technique that we propose avoids any parameter management while ensuring good theoretical convergence results. In our approach we do not need any complicated strategy to update the smoothing parameter $ r $ since we will consider it as a new variable. Our methods are validated by extensive numerical tests, in which we compare our methods to several other classical methods.
\end{abstract}
\maketitle

\renewcommand{\thefootnote}{}
\footnotetext{ $^*$Corresponding author: el-hassene.osmani@insa-rennes.fr }

\section{Introduction}

The linear complementarity problem consists in finding a vector in a finite-dimensional real vector space that satisfies a certain system of inequalities. Specifically, given a vector $ q  \in \R^{n}$ and a matrix $ M\in \R^{n\times n}$, the linear complementarity problem, abbreviated LCP, is to find a vector $ x \in \R^{n} $ such that 
\begin{equation} \label{1}
0 \leq x \perp (Mx+q) \geq 0. 
\end{equation}
This problem is known to have a unique solution for any $ q \in \R^{n} $ if and only if $ M $ is a P-matrix \cite{Cottle, Samelson}.
The linear complementarity problem has many important applications in engineering and equilibrium modeling \cite{Ferris, Pang}, and many numerical methods are developed to solve LCPs \cite{Billups, Chen}. Although the effectiveness of complementarity algorithms has improved substantially in recent years, the fact remains that increasingly more difficult problems are being proposed that are exceeding the capabilities of these algorithms. As a result, there is a real need to
propose new methods and algorithms to address complicated and difficult situations. Many algorithms have been proposed to solve problem LCP  \cite{Cottle, Murty}. They may be based on pivoting techniques \cite{CottleR, Lemke}, which often suffer from the combinatorial aspect of the problem, on interior point methods, which originate from an algorithm introduced by Karmarkar in linear optimization \cite{Karmarkar}, see also \cite{Kojima} for one of the first accounts on the use of interior-point methods to solve LCP. Some researchers try to solve LCPs by reformulating them as an unconstrained optimization \cite{Geiger}, and on nonsmooth Newton approaches \cite{Facchinei}, and rewrite the complementarity conditions as a system of smooth equations \cite{ill-posed}, such as the one considered here. See \cite{Cottle, Verlag} for other iterative methods. \\
In this work, we propose two new algorithms called TLCP and Soft-LCP for solving the LCP. The principle of these algorithms are as follows: first, we proposed two smoothing technique to regularize the complementary condition, we replace 
\begin{equation*}
0 \leq x \perp z \geq 0
\end{equation*}
by 
\begin{equation*}
\theta_{r}(x)+\theta_{r}(z)=1, \quad r\searrow 0,  
\end{equation*}
and 
\begin{equation*}
\forall \rho >0~~\quad~~  ~\quad~~~~~~~ x=\max (0, x-\rho z)\simeq r\log\left(1+e^{\dfrac{x-\rho z}{r}}\right), \quad r\searrow 0,
\end{equation*}
where $\theta_{r} $, $ \log $, and $ e^{.} $ operate componentwise on $ x $ and $ z $; then we give a strategy that decreases $ r $ during iterations and ensures the nonnegatives of variables. The main difference in our approach is that we do not need any complicated strategy to update the parameter $ r $ since we will consider it as a new variable. Finally, the two new algorithms are solved using the standard Newton method. To enforce a global convergence behavior, we also recommend using Armijo's line search.\\ 
This article is structured as follows. The first part of this paper is devoted to the presentation of the problem and gives some definitions and properties of the smoothing functions. In section 3, we give our approximate formulations and give the new formulation of the problem LCP. In section 4, we propose two new methods to solve the  LCP. In section 5,  we propose two generic algorithms to solve LCP and prove some convergence results. In section 6, we provide some numerical results where we present a comparison on some randomly generated problems and we study two concrete examples, the first one is a second-order ordinary differential equation and the second is an obstacle problem also, we tested our algorithms on several absolute value equations problems. Finally, we conclude our paper.
 \section{Preliminaries and Problem Setting }
In this section, we present some necessary definitions and lemmas. A matrix $ M\in \R^{n \times n} $ is said to be positive definite if $ \langle x, Mx  \rangle  >0$ for all nonzero $ x\in \R^{n}.$ $ M\in \R^{n \times n}  $  is called a $ \mathcal{P}$-matrix if all its minors are positive. As a consequence, if $ M $ is positive definite, then $ M $ is a  $ \mathcal{P}$-matrix. \\
Consider the linear complementarity problem LCP, which is to find a solution of the system $F(\mathbf{X})=0,$ with
\begin{equation}\label{2}
F(\mathbf{X})=\left[
\begin{array}{llllll} 
Mx+q-z\\
x.z 
\end{array}
\right],
\end{equation}
where $ \mathbf{X}=(x,~z)\in \R^{2n}_{+} $. Recall that the Hadamard product  $x.z  $ of two vectors $x  $ and $ z $ is the vector having its $ i $th component equal to $x_{i}z_{i}  $. \\To solve LCP, there are essentially three different classes of methods: equation-based methods (smoothing), merit functions and projection-type methods. Our goal in this paper is to present new and very simple smoothing and approximation schemes to solve LCP and to produce efficient numerical methods.
First, we state a result for the unique solution of an LCP, the following result was proved by Cottle, Pang and Stone \cite{Cottle}. Next, we  give the definition of $ \theta $-smoothing function and Soft-Max function
\begin{theorem} (Theorem 3.3.7, \cite{Cottle}). 
	A matrix $ M\in \R^{n \times n} $ is a  $ \mathcal{P}$-matrix if and only if the LCP \eqref{1} has a unique solution for every $ q \in \R^{n} $. 
\end{theorem}
\subsection{Definition of $ \theta $-smoothing function}
We introduce the function $ \theta $ with the following properties (these functions were used in \cite{M. Haddou and P. Maheux, HADDOU}). \\
Let $ \theta : \R \to ]-\infty,1[,$ be a non-decreasing continuous smooth concave function such that 
\begin{equation*}
\theta(t)<0  ~~\text{if}~~  t<0,~ \theta(0)=0 ~~\text{and}~~ \lim_{t \to +\infty} \theta(t)=1. 
\end{equation*}
One possible way to build such function is to consider non-increasing probability density functions $ f:\R_{+}\to \R_{+ } $ and then take the corresponding cumulative distribution function 
\begin{equation*}
\theta(t)=\int_{0}^t f(x)dx.
\end{equation*}
By definition of $f  $ we can verify that
\begin{equation*}
\lim_{t \to +\infty} \theta(t)= \int_{0}^{+\infty} f(x)dx=1,
\end{equation*}
and
\begin{equation*}
\theta(0)=\int_{0}^{0} f(x)dx=0. 
\end{equation*}
The non-decreasing hypothesis gives the concavity of $ \theta $. We then extend this functions for negative values in a smooth way.\\
Example of this family are $ \theta^{1}(t)=t/(t+1) $ if $ t\geq 0 $ and $ \theta^{1}(t)=t $ if $ t<0 $. \\
We introduce $ \theta_{r}(t):=\theta(\frac{t}{r}) $ for $ r>0.$ This definition is similar to the perspective functions in convex analysis. This functions satisfy 
\begin{center}
	$ \theta_{r}(0)=0 ~~  \forall r>0 $ and $ \lim_{r \searrow 0} \theta(t)= 1 ~~ \forall t >0.$
\end{center}
There are some examples of such functions 
\begin{equation*}
\theta_{r}^{1}(t)=\dfrac{t}{t+r}  ~\text{if} ~ t \geq 0  ~~ ~~\text{and}~~~~  \theta_{r}^{1}(t)=t ~\text{if}~  t<0,
\end{equation*}
\begin{equation*}
~~~~~~~~~~~~\theta_{r}^{2}(t)=1-e^{-t/r},~t\in \R.~~~~~~~~~~~~~~~~~~~~~~~~~~~~~~~~~~~~~~~~ 
\end{equation*}
The function $ \theta_{r}^{1}~$ will be extensively used in this paper and is illustrated in Figure. 1  for several values of $ r.$
\begin{figure}[H]\label{fig1}
	\begin{minipage}[H]{.46\linewidth}
		\begin{center}
			\includegraphics[width=9cm,height=7cm]{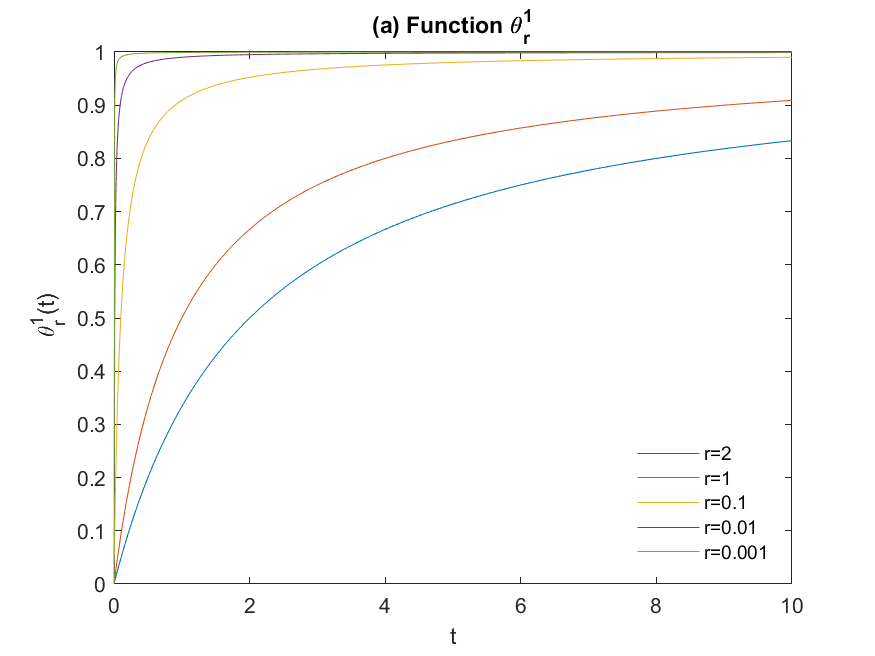}
		\end{center}
	\end{minipage} \hfill
	\begin{minipage}[H]{.46\linewidth}
		\begin{center}
			\includegraphics[width=9cm,height=7cm]{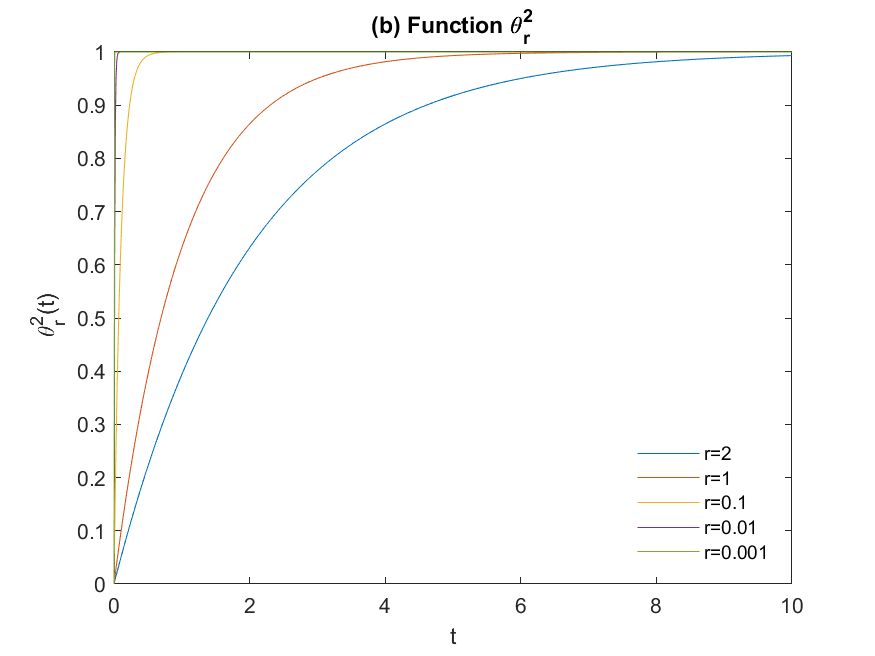}
		\end{center}
	\end{minipage}
	\caption{Function $ \theta_{r}$ for several values of $ r $.}
\end{figure}
It can be seen on the Figure. 1 that the function $ \theta_{r} $ behave as a step function when $ r $ becomes small.
\subsubsection{$ \theta $-smoothing of a complementarity condition}
A $\theta$-smoothing function paves the way for a smooth approximation of a complementarity condition. Let $ (x,z)\in \R^{2}$ be two scalars such that
\begin{equation} \label{2323}
0 \leq x \perp z \geq 0, 
\end{equation}
that is,
\begin{equation*}
x\geq 0,~~~~ z\geq 0, ~~~~ xz=0.  
\end{equation*}
In the $ (x,z)$-plane, the set of points obeying \eqref{2323} is the union of the two semi-axes  $ \{x \geq 0,~z=0\}~$ and \\$ \{x=0,~ z\geq 0\}.$ Visually, the nonsmoothness of \eqref{2323} is manifested by the "kink" at the corner $ (x,z)=(0,0).$ \\We consider two possible smooth approximations of \eqref{2323}, depending how it is rewritten in terms of $ \theta$-function.
\begin{lemma}
	\cite{HADDOU}	Given $ x, z \in \R_{+} $ and the parameter $ r>0,$ we have the equivalence  
	\begin{equation*}
	xz=0 \iff  \lim_{r \searrow 0} (\theta_{r}(x) + \theta_{r}(z)) \leq 1. 
	\end{equation*}
\end{lemma}
\begin{lemma}\label{120967}
	\cite{HADDOU}	$ \theta_{r}~$ is sub-additive for non-negative values, i.e. given $ x,~z \geq 0 $ it holds that 
	\begin{equation*}
	\theta_{r}(x) + \theta_{r}(z)~\geq~\theta_{r}(x+z) 
	\end{equation*}
	and we have the equivalence for $ r>0 $ 
	\begin{equation*}
	xz=0 \iff   \theta_{r}(x) + \theta_{r}(z)=\theta_{r}(x+z). 
	\end{equation*}
\end{lemma}
In the case of the function $ \theta^{1}_{r} $ and by definition of this function we have 
\begin{equation*}
\theta^{1}_{r}(x)+\theta^{1}_{r}(z)=1 \iff xz=r^{2}.
\end{equation*}
So, when $r$ goes to $ 0,$ we simply get $ xz=0 $.\\
Now, we define the Soft-Max function that we will use in the next section to approximate the complementarity condition.
\subsection{Soft-Max Function}
Let $ f $ be a function defined as:
\begin{equation*}
f(x_{1},...,x_{n})=\max(x_{1},...,x_{n}),	
\end{equation*}
obviously, the max function is non-differentiable. We approximate the max function by a smooth function, noted Soft-Max function as introduced in \cite{Soft} by: 
\begin{equation*}
g_{r}(x_{1},...,x_{n})=r\log\left(~\sum_{i=1}^{n}e^{~x_{i}/r}\right). 
\end{equation*}
Indeed:
\begin{equation*}
g_{r}(x)=r\log\left(~\sum_{i=1}^{n}e^{~x_{i}/r}\right) -r\log n,
\end{equation*}
then $\forall r>0 $ and  $\forall x \in \R^{n},$
\begin{equation*}
\begin{split}
g_{r}(x)  \leq r\log\left(~n\max_{i}e^{~x_{i}/r}\right)-r\log n =& \max_{i}x_{i},
\end{split}	
\end{equation*}
\begin{equation*}
\begin{split}
\max_{i}x_{i}\leq r\log\left(~\sum_{i=1}^{n}e^{~x_{i}/r}\right)=&g_{r}(x)+r\log n.
\end{split}	
\end{equation*}
Then 
\begin{equation*}
\Vert g_{r}(x)-\max_{i} x_{i}    \Vert \leq r\log n.
\end{equation*}
Thus $ g_{r}$ is a uniformly smoothing approximation function of $ f.$ Notice that the accuracy of the Soft-Max approximation depends on scale $ r $. 
\section{An approximate formulation}
In this section, we present two formulations for LCP \eqref{1} by using two approximations, the first with the $ \theta$-function and the second with the Soft-Max function.
\subsection{ Approximation of LCP using $ \theta $-function}
We reformulate the problem LCP using $ \theta_{r} $ function, we regularize each complementarity constraint by considering 
\begin{equation*}
x_{i}z_{i}=0, ~~~ \text{by} ~~~ \theta_{r}(x_{i})+\theta_{r}(z_{i})=1,~~~~~~~ \forall~  i=1,...n, 
\end{equation*}
in fact $ x_{i}z_{i}=0 $ should be approximated by 
\begin{equation*}
\theta_{r}(x_{i})+\theta_{r}(z_{i}) \leq 1, ~~~~ \text{(both can be zeros)}
\end{equation*}
but we use an implicit assumption  of strict complementarity. Using this approximation we obtain the following formulation:
\begin{equation}
(\tilde{P}_{\theta})~~~~~~\left\{
\begin{array}{llllll} 
Mx+q=z,  \\
x \geq 0,~~~ z \geq 0,~~~ r\searrow 0 \\
\theta_{r}(x)+\theta_{r}(z)-\mathbf{1}=0.
\end{array}
\right.
\end{equation}
Here, it is understood that $ \theta_{r}$ operates componentwise on $ x$ and $z,$ while $ \mathbf{1}\in \R^{n}~ $ is the vector whose entries are all equal to $ 1.~ $
We consider the family $ \{\tilde{F}_{\theta}(.,r),~r>0\}$,  where 
\begin{equation}\label{10}
\tilde{F}_{\theta}(\mathbf{X},r)=\left[
\begin{array}{llllll} 
Mx+q-z\\
r(\theta_{r}(x)+\theta_{r}(z)-\mathbf{1}) 
\end{array}
\right],
\end{equation}
is a regularized function of $ F $ defined in \eqref{2}. It is highly recommended that the smoothed complementarity equations in \eqref{10} be premultiplied by $ r $, so as to control the magnitude of their partial derivatives. Indeed, for all $ t \geq 0, $
\begin{equation*}
\theta_{r}^{'}(t)=\frac{1}{r}\theta^{'}\left(\frac{t}{r}\right), 
\end{equation*} 
can be seen to blow up when $ r\downarrow 0,~ $while $ r\theta^{'}_{r}(t) $ tends to the finite limit $ \theta^{'}(0). $ 
\subsection{Approximation of LCP using Soft-Max}
It is obvious that the vectors $ x $ and $ z $ satisfy complementarity condition if and only if
\begin{equation*}
\begin{split}
\forall \rho >0~~\quad~~& ~\quad~~~~~~~ x=\max (0, x-\rho z).
\end{split}	
\end{equation*}
Using the Soft-Max function  defined below, we obtain an approximate formulation for LCP 
\begin{equation}
(P_{s}^{r})~~~~~~\left\{
\begin{array}{llllll} 
Mx+q=z, ~~~ r\searrow 0, ~~~\rho >0 \\
x=\max(0,x-\rho z)\simeq r\log\left(1+e^{\dfrac{x-\rho z}{r}}\right),
\end{array}
\right.
\end{equation}
by the same way as for \eqref{10}, $ \log $ and $ e^{.} $ operate componentwise on $ x $ and $ z $. We consider the family $ \{\tilde{F}_{s}(.,r),~r>0\},$ where 
\begin{equation}\label{11}
\tilde{F}_{s}(\mathbf{X},r)=\left[
\begin{array}{llllll} 
Mx+q-z\\
x-r\log\left(1+e^{\dfrac{x-\rho z}{r}}\right) 
\end{array}
\right],
\end{equation}	
is a regularized function of $ F $ defined in \eqref{2}.
\section{Solving LCP via New Algorithm}
In this section, we present the idea of our algorithms for optimization problems to solve the LCP, but here we don't have any objective function to minimize. Our methods take inspiration from Interior Point Methods. \\
We recall that the interior-point methods have replaced the original nonsmooth problem LCP by a sequence of regularized problems
\begin{equation} \label{12345543211234}
F_{r}(\mathbf{X})=0,
\end{equation}
where 
\begin{equation}
\mathbf{X}=\left [
\begin{array}{llllll} 
x \\
z
\end{array}\right ] \in \R^{2n}_{+}, \quad F_{r}(\mathbf{X})=\left [
\begin{array}{llllll} 
Mx+q-z \\
x.z-r\mathbf{1}
\end{array}\right ],
\end{equation}
where $ r \geq 0 $ is the smoothing parameter, $\mathbf{1}  \in \R^{n}$ is the vector whose components are all equal to 1.
The Jacobian matrix of $ F_{r} $ with respect to $ \mathbf{X}$, does not depend on $ r $ and can be denoted by   
\begin{equation} \label{0854356}
\nabla_{\mathbf{X}}F_{r}(\mathbf{X})=\left(\begin{array}{cc}
M&-I\\
Z& X
\end{array}\right),
\end{equation}
where $ Z=\text{diag}(z) $ and  $ X=\text{diag}(x)$, i.e. the diagonal matrix of $ z $ (resp. $ x $).
\subsection{When the parameter becomes a variable}
In the system \eqref{12345543211234}, the status of the parameter $ r $ is very distinct from that of the variable $ \mathbf{X}$. While $ \mathbf{X}$ is computed "automatically" by a Newton iteration, $ r$ has to be updated "manually" in an ad-hoc manner.  \\
Our goal is to find a strategy that decreases $ r $ during iterations and ensures the nonnegative of variables. However, we must adjust the strategy when the model or its parameters are changed. 
To avoid this trouble, we consider $ r $ as an unknown of the system instead of a parameter. We feel that it would be judicious to incorporate the parameter $ r$ into the variables.
Let us therefore consider the enlarged vector of unknowns 
\begin{equation}
\mathbb{X}= \left [\begin{array}{llllll} 
\mathbf{X} \\
r
\end{array}\right] \in \R^{2n} \times \R_{+},
\end{equation}
and then consider a system  of $ 2n+1 $ equations
\begin{equation}
\mathbb{F}_{\theta}(\mathbb{X})=0, ~~(\text{resp}.~~ \mathbb{F}_{s}(\mathbb{X})=0),
\end{equation}
to be on $\mathbb{X}$. To this end, let us remind ourselves that our ultimate goal is to solve $ \tilde{F}_{\theta}(\mathbf{X},0)=0$ (resp. $ \tilde{F}_{s}(\mathbf{X},0)=0$), together with the inequalities $ x\geq 0,~z\geq 0. $ \\
To ensure the nonnegative of variables, we need a new equation. So we consider the function $ f(u)=\frac{1}{2}\text{min}^{2}(u, 0).$ It is easy to see that $ f(u)=0 $ when $ u \geq 0 $. Then we add the  following equation in $ \tilde{F}_{\theta}(\mathbf{X},r)$ (resp. $ \tilde{F}_{s}(\mathbf{X},r)$)
\begin{equation} \label{12345}
\frac{1}{2} \Vert x^{-}\Vert^{2}+ \frac{1}{2} \Vert z^{-}\Vert^{2}+r^{2}=0,
\end{equation}
where
\begin{equation*}
\Vert x^{-}\Vert^{2}=\sum_{i=1}^{n} \text{min}^{2}(x_{i},0), \quad \Vert z^{-}\Vert^{2}=\sum_{i=1}^{n} \text{min}^{2}(z_{i},0).
\end{equation*} 
This equation implies that $r=0$ and all variables are nonnegative. Hence, we can define the new system using $ \theta$-function (we restrict our choice of $ \theta$-function to  $ \theta_{r}^{1}(x)=\frac{x}{x+r} $.) and Soft-Max function by 
\begin{equation} \label{889}
\mathbb{F}_{\theta}(\mathbb{X})=\left[\begin{array}{llllll} 
Mx+q-z \\
r(\theta_{r}^{1}(x)+\theta_{r}^{1}(z)-\mathbf{1})\\
\frac{1}{2} \Vert x^{-}\Vert^{2}+\frac{1}{2} \Vert z^{-}\Vert^{2}+r^{2}
\end{array}
\right],
\end{equation}
and 
\begin{equation} \label{88}
\mathbb{F}_{s}(\mathbb{X})=\left [
\begin{array}{llllll} 
Mx+q-z \\
x-r\log\left(1+e^{\dfrac{x-\rho z}{r}}\right) \\
\frac{1}{2} \Vert x^{-}\Vert^{2}+\frac{1}{2} \Vert z^{-}\Vert^{2}+r^{2}
\end{array}
\right].
\end{equation}
The two Jacobian matrices of $\mathbb{F}_{\theta} $ and $ \mathbb{F}_{s} $ are:  
\begin{equation}
\nabla_{\mathbb{X}}\mathbb{F}_{\theta}(\mathbb{X})=
\begin{pmatrix}
M_{n\times n} & -I_{n \times n} & 0_{n \times 1}  \\
\text{diag}\left(\dfrac{r^{2}}{(x+r)^{2}}\right)& \text{diag}\left(\dfrac{r^{2}}{(z+r)^{2}}\right)& W \mathbf{e} \\
(x^{-})^{\text{T}}&(z^{-})^{\text{T}}& 2r  \\
\end{pmatrix},
\end{equation}
and 
\begin{equation}
\nabla_{\mathbb{X}}\mathbb{F}_{s}(\mathbb{X})=\begin{pmatrix}
M_{n\times n} & -I_{n \times n} & 0_{n \times 1}  \\
\text{diag}\left(\dfrac{1}{1+e^{\dfrac{x-\rho z}{r}}}\right)&\text{diag}\left(\dfrac{\rho e^{\dfrac{x-\rho z}{r}}}{1+e^{\dfrac{x-\rho z}{r}}}\right) &V\mathbf{e}   \\
(x^{-})^{\text{T}} & (z^{-})^{\text{T}} & 2r 	
\end{pmatrix},
\end{equation}
where $  x^{-}  $ is the vector of components $ x_{i}^{-}=\min(x_{i},0) $ and similarly for $ z^{-},$
\begin{equation*}
V =\text{diag}\left(-\log(1+e^{\dfrac{x-\rho z}{r}})+\dfrac{\dfrac{x-\rho z}{r}e^{\dfrac{x-\rho z}{r}}}{1+e^{\dfrac{x-\rho z}{r}}}    \right),
\end{equation*}
and 
\begin{equation*}
W=\text{diag} \left(\dfrac{x^{2}}{(x+r)^{2}}+\dfrac{z^{2}}{(z+r)^{2}}-1\right),
\end{equation*}
and $ \mathbf{e}  $ is a n-dimensional vector whose entries are equal to $ 1 $. If $ \mathbb{F}_{\theta}(\mathbb{X})=0$ (resp. $\mathbb{F}_{s}(\mathbb{X})=0$) we obtain $ r=0. $ Hence in this case, $ \nabla_{\mathbb{X}}\mathbb{F}_{\theta}(\mathbb{X})$ becomes singular (resp. $ \nabla_{\mathbb{X}}\mathbb{F}_{s}(\mathbb{X}) $ becomes singular) since $ \text{det} \nabla_{\mathbb{X}}\mathbb{F}_{\theta}(\mathbb{X})=0 $ (resp. $ \text{det}\nabla_{\mathbb{X}}\mathbb{F}_{s}(\mathbb{X})=0$). To solve this issue, we add a small enough positive parameter $ \varepsilon $ to equation \eqref{12345}. We get 
\begin{equation} \label{123456}
\frac{1}{2} \Vert x^{-}\Vert^{2}+ \frac{1}{2} \Vert z^{-}\Vert^{2}+r^{2}+\varepsilon r=0.
\end{equation} 
Hence, we define the following systems
\begin{equation} \label{889}
\mathbb{F}_{\theta}(\mathbb{X})=\left[\begin{array}{llllll} 
Mx+q-z \\
r(\theta_{r}^{1}(x)+\theta_{r}^{1}(z)-\mathbf{1})\\
\frac{1}{2} \Vert x^{-}\Vert^{2}+ \frac{1}{2} \Vert z^{-}\Vert^{2}+r^{2}+\varepsilon r
\end{array}\right],
\end{equation}
and 
\begin{equation} \label{88}
\mathbb{F}_{s}(\mathbb{X})=\left [
\begin{array}{llllll} 
Mx+q-z \\
x-r\log\left(1+e^{\dfrac{x-\rho z}{r}}\right) \\
\frac{1}{2} \Vert x^{-}\Vert^{2}+ \frac{1}{2} \Vert z^{-}\Vert^{2}+r^{2}+\varepsilon r
\end{array}
\right ].
\end{equation}
\section{Convergence}
In this section, we propose two generic algorithms to solve LCP and prove some convergence results. First, we present two Lemmas that will used to prove our main results.
\begin{lemma} \label{22786633}
	We consider the following system 
	\begin{equation} 
	\begin{array}{llllll}
	Z.X=0 \\
	Z \geq 0,~~ X\geq 0 ,
	\end{array}
	\end{equation}		
	where $ Z=\text{diag}(z) $ and  $ X=\text{diag}(x)$. Assume that $ Z,~X$ are strictly complementary (i.e. $ \exists~ \alpha >0 $ such that  $ Z+X>\alpha$). Then $ J $ is singular if and only if $ J_{l} $ is singular, where 
	\begin{equation*}
	J=\left(\begin{array}{cc}
	M&-I\\
	Z& X
	\end{array}\right) \text{and} ~~~~ 
	J_{l}=\left(\begin{array}{cc}
	M&-I\\
	\phi(Z)& \phi(X)
	\end{array}\right),
	\end{equation*}
	such that 
	\begin{equation*}
	\phi(t)=\left\{\begin{array}{ccc}
	1 &\text{if}& t \neq 0\\
	0 &\text{if}& t=0,\\
	\end{array} \right.
	\end{equation*}
	here $ \phi $ operates componentwise on $ t $, and it verifies the following system
	\begin{equation*} 
	\begin{array}{llllll}
	\phi(Z).\phi(X)=0 \\
	\phi(Z) \geq 0,~~ \phi(X)\geq 0.
	\end{array}
	\end{equation*}
\end{lemma}
\begin{proof}
	By the strict complementarity hypothesis, we range the rows and the columns of  $ J $ and  $ J_{l}$ as follows   
	\begin{displaymath}
	J_{\sigma}=\left(\begin{array}{cc}
	M_{\sigma}&-I_{\sigma}\\
	\left(\begin{array}{cc}
	Z_{1}&0\\
	0& 0
	\end{array}\right)& \left(\begin{array}{cc}
	0&0\\
	0& X_{2}
	\end{array}\right)
	\end{array}\right),
	\end{displaymath}
	where $ X_{2} > 0$ and $ Z_{1} >0 $ and 
	\begin{displaymath}
	(J_{l})_{\sigma}=\left(\begin{array}{cc}
	M_{\sigma}&-I_{\sigma}\\
	\left(\begin{array}{ccc}
	\begin{array}{ccc}
	1&        & \\
	& \ddots &  \\
	&        & 1 \\
	\end{array}&        & 0\\
	&  &  \\
	0 &        & 0 \\
	\end{array}\right)& \left(\begin{array}{ccc}
	
	0&        & 0\\
	& \ddots &  \\
	0&        & \begin{array}{ccc}
	
	1&        & 0\\
	& \ddots &  \\
	0&        & 1 \\
	\end{array} \\
	\end{array}\right)
	\end{array}\right).
	\end{displaymath}
	The determinant of the two matrices  $ J_{\sigma} $ and $ (J_{l})_{\sigma} $ are equal to 
	\begin{equation*}
	\text{det}(J_{\sigma})= \left\vert \begin{array}{cc}
	M_{\sigma}&-I_{\sigma}\\
	\left(\begin{array}{cc}
	Z_{1}&0\\
	0& 0
	\end{array}\right)& \left(\begin{array}{cc}
	0&0\\
	0& X_{2}
	\end{array}\right)
	\end{array}\right \vert =\pm \prod_{i \in \I}^{} x_{i} \prod_{i \in \I}^{} z_{i}~~ \text{det}(C),
	\end{equation*}
	\begin{equation*}
	\text{det}((J_{l})_{\sigma})=\left \vert\begin{array}{cc}
	M_{\sigma}&-I_{\sigma}\\
	\left(\begin{array}{ccc}
	\begin{array}{ccc}
	1&        & \\
	& \ddots &  \\
	&        & 1 \\
	\end{array}&        & 0\\
	&  &  \\
	0 &        & 0 \\
	\end{array}\right)& \left(\begin{array}{ccc}
	0&        & 0\\
	& \ddots &  \\
	0&        & \begin{array}{ccc}
	1&        & 0\\
	& \ddots &  \\
	0&        & 1 \\
	\end{array} \\
	\end{array}\right)
	\end{array}\right\vert=\pm \prod_{i \in \I}^{} \phi(x_{i}) \prod_{i \in \I}^{} \phi(z_{i})~~ \text{det}(C),
	\end{equation*}
	where $ C $ is a certain matrix and $ \I= \{1,...,n\}$. Since
	\begin{equation*}
	\pm \prod_{i \in I}^{} x_{i} \prod_{i \in \I}^{} z_{i} ~~~\quad\text{and} ~~~ \quad\prod_{i \in \I}^{} \phi(x_{i}) \prod_{i \in \I}^{} \phi(z_{i}),
	\end{equation*}
	are nonzeros, then we can conclude that $ J $ and $ J_{l} $ are invertibles and singulars at the same time.
\end{proof}
\begin{lemma} \label{19624}
	Suppose that   $ M $ has all its principal minors are nonzeros. Then, $ J $ is invertible, where 
	\begin{equation*}
	J=\left(\begin{array}{cc}
	M&-I\\
	Z& X
	\end{array}\right).
	\end{equation*}
\end{lemma}
\begin{proof}
	Suppose that $ M $ is decomposed as  
	\begin{displaymath}
	M=\left(\begin{array}{cc}
	M_{11}&M_{12}\\
	M_{21}& M_{22}
	\end{array}\right).
	\end{displaymath}
	By Lemma \eqref{22786633}, the determinant of $ (J_{l})_{\sigma} $ is equal to 
	\begin{equation*}
	\begin{split}
	\vert (J_{l})_{\sigma} \vert =\left\vert\begin{array}{cccccccccc}
	M_{11}&      &    && M_{12}    &  -1    &     &         &   &0    \\
	&      &    &       &    &        &     &  \ddots&    &   \\
	M_{21}&      &    & &M_{22}    &    0   &    &        & \ddots  &  \\
	&      &    &      &    &        &     &      &  & -1\\
	1     &&    &       & 0   &    0   &    &        &   &0  \\
	&   \ddots   &   &       &    &        &     &       &   &     \\
	&      &    1&       &    &       &     &       &   &     \\
	&      &    &       &    &        &    & 1 &   &   \\
	&      &    &      &   &      &    &       & \ddots  &      \\
	0&      &    &      &0    &    0   &    &       &   &1      \\
	\end{array}\right\vert &=\pm \left\vert\begin{array}{ccc}
	\begin{array}{ccc}
	M_{11}&~ &  \\
	& ~&\\
	& ~& \\
	\end{array}&\begin{array}{ccc}
	M_{12}\\
	\\
	\\
	\end{array}&\begin{array}{ccc}
	-1&        & \\
	& \ddots &  \\
	&        &- 1 \\
	\end{array}\\
	M_{21}~~~~~~~~~& M_{22}&0 \\
	\begin{array}{ccc}
	1&        & \\
	& \ddots &  \\
	&        & 1 \\
	\end{array}&  \begin{array}{ccc}
	\\
	\\
	0\\
	\end{array}   &   \begin{array}{ccc}
	\\
	\\
	0\\
	\end{array}    
	\end{array}\right\vert \\
	&=\pm\left\vert\begin{array}{cc}
	\begin{array}{ccc}
	M_{12} \\
	\\
	\\
	\end{array}   &\begin{array}{ccc}
	-1&        & \\
	& \ddots &  \\
	&        &- 1 \\
	\end{array}\\
	M_{22}&0 \\    
	\end{array}\right\vert=\pm \vert M_{22} \vert.
	\end{split}
	\end{equation*}
	In view of Lemma \eqref{22786633}, we can conclude that $ J $ is invertible. 
\end{proof}

Below is a result about the Jacobian matrix of $ \mathbb{F}_{s}(\mathbb{X}),$ and which will be useful for later purposes.
\begin{theorem} \label{765699}
	Suppose that $\mathbf{X}^{*}=(x^{*},z^{*})$ be a solution of LCP, $ \nabla_{\mathbf{X}}{F}_{0}(\mathbf{X}^{*})$ \eqref{0854356} is invertible and $ \mathbf{X}^{*} $ verifies the strict complementarity (i.e. $ \exists~ \alpha >0 $  such that $ x^{*}_{i}+z^{*}_{i}>\alpha$, $\forall i \in \{1, ..., n\} $). Then  $ \lim\limits_{\substack{r \rightarrow 0}}\nabla_{\mathbb{X}}	\mathbb{F}_{s}(\mathbf{X}^{*},r)  $ is invertible, i.e. the two Jacobian matrices are singular or nonsigular at the same time.
\end{theorem}
\begin{proof}
	Let 
	\begin{equation*} 
	\mathbb{F}_{s}(\mathbb{X})=\left [
	\begin{array}{llllll} 
	\mathbb{F}_{s,~1}(\mathbb{X}) \\
	\mathbb{F}_{s,~2}(\mathbb{X}) \\
	\mathbb{F}_{s,~3}(\mathbb{X})
	\end{array}
	\right ]=\left [
	\begin{array}{llllll} 
	Mx+q-z \\
	x-r\log\left(1+e^{\dfrac{x-\rho z}{r}}\right) \\
	\frac{1}{2} \Vert x^{-}\Vert^{2}+ \frac{1}{2} \Vert z^{-}\Vert^{2}+r^{2}+\varepsilon r
	\end{array}
	\right ].
	\end{equation*}
	Since $ r $ is now considered as a variable and  the scalar function $ t \mapsto \frac{1}{2}  \vert \min(t,0)\vert^{2} $ is differentiable and its derivative is equal to $ \min(t,0) $. From this observation, it follows that 
	\begin{equation*}
	\nabla_{\mathbb{X}}	\mathbb{F}_{s}(\mathbb{X})=
	\begin{pmatrix}
	M_{n\times n} & -I_{n \times n} & 0_{n \times 1}  \\
	\nabla_{x} \mathbb{F}_{s,2}(\mathbb{X})	& \nabla_{z} \mathbb{F}_{s,2}(\mathbb{X})& \nabla_{r} \mathbb{F}_{s,2}(\mathbb{X}) \\
	(x^{-})^{\text{T}}&(z^{-})^{\text{T}}& 2r+\varepsilon  \\
	\end{pmatrix},
	\end{equation*}
	where $ x^{-} $ is the vector of components $ x_{i}^{-}=\min(x_{i},0) $ and similarly for $ z^{-}$. Since $ \mathbf{X}^{*}=(x^{*},z^{*}) $ is a solution of LCP, we have
	
	\begin{enumerate}
		\item The derivative of $ \mathbb{F}_{s,~2}(\mathbf{X},r) $ with respect to $ x$  is:
		\begin{equation*}
		\nabla_{x} \mathbb{F}_{s,2} (x^{*},z^{*},r)=\text{diag}\left(\dfrac{1}{1+e^{\dfrac{x^{*}-\rho z^{*}}{r}}}\right)_{n \times n}, 
		\end{equation*}
		when $ r $ goes to $ 0 $ and in view of the strict complementary we have to consider two cases:
		\begin{itemize}
			\item  $ x_{i}^{*}\to 0, $ and $ z_{i}^{*}>0 $ $\forall i \in \{1, ..., n\} $ then 
		\end{itemize}
		$ ~~~~~$   $ \lim\limits_{\substack{r \rightarrow 0 \\ x^{*}_{i}\rightarrow 0}}\nabla_{x} 	\mathbb{F}_{s,2} (x_{i}^{*},z_{i}^{*},r)=  \lim\limits_{\substack{r \to 0}}\dfrac{1}{1+e^{-\frac{\rho z_{i}^{*}}{r}}}=1.$
		\begin{itemize}
			\item  $ x_{i}^{*}>0, $ and $ z_{i}^{*} \to 0 $ $\forall i \in \{1, ..., n\} $ then 
		\end{itemize}
		$ ~~~~~$   $ \lim\limits_{\substack{r \rightarrow 0 \\ z^{*}_{i}\rightarrow 0}} \nabla_{x}\mathbb{F}_{s,2} (x_{i}^{*},z_{i}^{*},r)=\lim\limits_{\substack{r \rightarrow 0 }}\dfrac{1}{1+e^{\frac{ x_{i}^{*}}{r}}}=0.$
		\item The derivative of  $ \mathbb{F}_{s,~2}(\mathbf{X},r) $ with respect to $ z$ is: 	
		\begin{equation*}
		\nabla_{z} \mathbb{F}_{s,2} (x^{*},z^{*},r)=\text{diag}\left(\dfrac{\rho e^{\dfrac{x^{*}-\rho z^{*}}{r}}}{1+e^{\dfrac{x^{*}-\rho z^{*}}{r}}}\right)_{n \times n},   
		\end{equation*}
		when $ r $  goes to $ 0 $ and in view of the strict complementary we have to consider two cases:
		\begin{itemize}
			\item  $ x_{i}^{*} \to 0, $ and $ z_{i}^{*}>0 $ $\forall i \in \{1, ..., n\} $ then 
		\end{itemize}
		$ ~~~~~$   $ \lim\limits_{\substack{r \rightarrow 0 \\ x^{*}_{i}\rightarrow 0}}\nabla_{z} 	\mathbb{F}_{s,2} (x_{i}^{*},z_{i}^{*},r)=  \lim\limits_{\substack{r \rightarrow 0}}\rho \dfrac{e^{-\frac{\rho z_{i}^{*}}{r}}}{1+e^{-\frac{\rho z_{i}^{*}}{r}}}=0.$
		\begin{itemize}
			\item  $ x_{i}^{*}>0, $ and $ z_{i}^{*} \to 0 $ $\forall i \in \{1, ..., n\} $ then 
		\end{itemize}
		$ ~~~~~$   $ \lim\limits_{\substack{r \rightarrow 0 \\ z^{*}_{i}\rightarrow 0}}\nabla_{z} 	\mathbb{F}_{s,2} (x_{i}^{*},z_{i}^{*},r)=  \lim\limits_{\substack{r \rightarrow 0}}\rho \dfrac{e^{\frac{ x_{i}^{*}}{r}}}{1+e^{\frac{x_{i}^{*}}{r}}}=\rho~~~$ ~~~ (thereafter, we fixed $ \rho=1 $).
		
		\item The derivative of  $ \mathbb{F}_{s,~2}(\mathbf{X},r) $  with respect to $ r$ is: 
		\begin{equation*}
		\nabla_{r}\mathbb{F}_{s,2}(x^{*},z^{*},r)=\left(-\log(1+e^{\dfrac{x^{*}-\rho z^{*}}{r}})+\dfrac{\dfrac{x^{*}-\rho z^{*}}{r}e^{\dfrac{x^{*}-\rho z^{*}}{r}}}{1+e^{\dfrac{x^{*}-\rho z^{*}}{r}}}    \right)_{n \times 1},
		\end{equation*}
		when $ r $  goes to $ 0 $ and in view of the strict complementary we have to consider two cases:
		\begin{itemize}
			\item $ x_{i}^{*} \to 0, $ and $ z_{i}^{*}>0 $  $\forall i \in \{1, ..., n\}$  then 
		\end{itemize}
		\begin{equation*}
		\lim\limits_{\substack{r \rightarrow 0 \\ x^{*}_{i}\rightarrow 0}}\nabla_{r} 	\mathbb{F}_{s,2} (x_{i}^{*},z_{i}^{*},r)
		=\lim_{r \to 0}\left[-\log(1+e^{-\frac{\rho z_{i}^{*}}{r}})-\frac{\rho z_{i}^{*}}{r}\dfrac{e^{-\frac{\rho z_{i}^{*}}{r}}}{1+e^{-\frac{\rho z_{i}^{*}}{r}}}\right]=0.
		\end{equation*}	
		\begin{itemize}
			\item  $ x_{i}^{*}>0, $ and $ z_{i}^{*} \to 0 $  $\forall i \in \{1, ..., n\}$  then 
		\end{itemize}
		\begin{equation*}
		\lim\limits_{\substack{r \rightarrow 0 \\ z^{*}_{i}\rightarrow 0}}\nabla_{r} 	\mathbb{F}_{s,2} (x_{i}^{*},z_{i}^{*},r)=\lim_{r \to 0}\left[-\log(1+e^{\frac{x_{i}^{*}}{r}})+\frac{x_{i}^{*}}{r}\dfrac{e^{\frac{ x_{i}^{*}}{r}}}{1+e^{\frac{ x_{i}^{*}}{r}}}\right]=0.
		\end{equation*}	
	\end{enumerate}
	Finally, thanks to the assumption $ \mathbf{X}^{*}=(x^{*},z^{*}) $ is a solution of LCP, we have $ x^{*}\geq 0 $ and  $ z^{*}\geq 0 $, so that $ x^{-}=z^{-}=0. $ Hence 
	
	\begin{equation*}
	\lim_{r \to 0}\nabla_{\mathbb{X}}	\mathbb{F}_{s}(\mathbf{X}^{*},r) =\begin{bmatrix}
	\begin{pmatrix}
	M_{n\times n} & -I_{n \times n}  \\
	\phi(Z^{*}) & \phi(X^{*}) \\
	\end{pmatrix}
	&        
	\begin{matrix}
	0 \\[3mm]
	0 \\[3mm]
	\end{matrix}
	\\ 
	0\quad  \quad  \quad 0      &      \varepsilon  \\
	\end{bmatrix},
	\end{equation*}
	and 
	\begin{equation*}
	\lim_{r \to 0} \left\vert\nabla_{\mathbb{X}}	\mathbb{F}_{s}(\mathbf{X}^{*},r)\right 
	\vert   
	=\varepsilon \left\vert
	\left(\begin{array}{cc}
	M&-I\\
	\phi(Z^{*})& \phi(X^{*})
	\end{array}\right) \right 
	\vert,
	\end{equation*}
	where $\phi(.)$ is defined in Lemma \eqref{22786633}, $ Z^{*}=\text{diag}(z^{*}) $ and $ X^{*}=\text{diag}(x^{*})$.\\
	In view of Lemma \eqref{22786633}, we conclude that if  $ \nabla_{\mathbf{X}}F_{0}(\mathbf{X}^{*})$ is invertible  then $\lim\limits_{\substack{r \rightarrow 0}}\nabla_{\mathbb{X}}\mathbb{F}_{s}(\mathbf{X}^{*},r)  $ is invertible.
\end{proof}
Here we present the same result but for the system $ \mathbb{F}_{\theta}(\mathbb{X})$. 
\begin{theorem}
	Suppose that $\mathbf{X}^{*}=(x^{*},z^{*})$ be a solution of LCP, $ \nabla_{\mathbf{X}}{F}_{0}(\mathbf{X}^{*})$ \eqref{0854356} is invertible and $ \mathbf{X}^{*} $ verifies the strict complementarity (i.e. $ \exists~ \alpha >0 $  such that $ x^{*}_{i}+z^{*}_{i}>\alpha$, $\forall i \in \{1, ..., n\} $). Then  $ \lim\limits_{\substack{r \rightarrow 0}}\nabla_{\mathbb{X}}	\mathbb{F}_{\theta}(\mathbf{X}^{*},r)  $ is invertible, i.e. the two Jacobian matrices are singular or nonsigular at the same time.
\end{theorem}
\begin{proof}
	As the theorem \eqref{765699}, we have to show that 
	\begin{equation*}
	\lim_{r \to 0}\nabla_{\mathbb{X}}	\mathbb{F}_{\theta}(\mathbf{X}^{*},r) =\begin{bmatrix}
	\begin{pmatrix}
	M_{n\times n} & -I_{n \times n}  \\
	\phi(Z^{*}) & \phi(X^{*})\\
	\end{pmatrix}
	&        
	\begin{matrix}
	0 \\[3mm]
	0 \\[3mm]
	\end{matrix}
	\\ 
	0\quad  \quad  \quad 0      &      \varepsilon  \\
	\end{bmatrix}.
	\end{equation*}
	Let us consider the set $ S $ defined as
	\begin{equation*}
	S=\{  (x_{i},z_{i},r)/ ~~ \theta^{1}_{r}(x_{i})+\theta^{1}_{r}(z_{i})=1,  \forall i \in \{1, ..., n\} \}, 	
	\end{equation*}
	by Lemma \eqref{120967}, we have 
	\begin{equation*}
	\theta^{1}_{r}(x_{i})+\theta^{1}_{r}(z_{i})=1 \iff x_{i}z_{i}=r^{2}, \quad  \forall i \in \{1, ..., n\}.
	\end{equation*}
	So, we have 
	\begin{equation*}
	S=\{ (x_{i},z_{i},r)/~~ \theta^{1}_{r}(x_{i})+\theta^{1}_{r}(z_{i})=1, \forall i \in \{1, ..., n\} \}=\{ (x_{i},z_{i},r)/~~ x_{i}z_{i}-r^{2}=0,\forall i \in\{1, ..., n\} \}. 
	\end{equation*}
	Since $ \mathbf{X}^{*}=(x^{*},z^{*}) $ is a solution of LCP, we deduce that  
	$(x^{*},z^{*},r)  $  is near to $ S $, then 
	\begin{equation*}
	x^{*}z^{*}-r^{2}= \smallO{r},
	\end{equation*}
	i.e. $ x^{*}z^{*}-r^{2}$ is negligent by $ r.$ In view of the assumption of the strict complementary we have to consider two  cases if $ z^{*}_{i} >0 $ then $ x^{*}_{i}=\smallO{r} $ and if $x^{*}_{i}>0$ then $ z^{*}_{i}=\smallO{r}.$ Let

	\begin{equation*} 
	\mathbb{F}_{\theta}(\mathbb{X})=\left [
	\begin{array}{llllll} 
	\mathbb{F}_{\theta,~1}(\mathbb{X}) \\
	\mathbb{F}_{\theta,~2}(\mathbb{X}) \\
	\mathbb{F}_{\theta,~3}(\mathbb{X})
	\end{array}
	\right ]=\left [
	\begin{array}{llllll} 
	Mx+q-z \\
	\dfrac{rx}{x+r}+\dfrac{rz}{z+r}-r\mathbf{1} \\
	\frac{1}{2} \Vert x^{-}\Vert^{2}+ \frac{1}{2} \Vert z^{-}\Vert^{2}+r^{2}+\varepsilon r
	\end{array}
	\right ].
	\end{equation*}
	The jacobian matrix of $\mathbb{F}_{\theta} $ is:
	\begin{equation*}
	\nabla_{\mathbb{X}}	\mathbb{F}_{\theta}(\mathbb{X})=
	\begin{pmatrix}
	M_{n\times n} & -I_{n \times n} & 0_{n \times 1}  \\
	\nabla_{x} \mathbb{F}_{\theta,2}(\mathbb{X})	& \nabla_{z} \mathbb{F}_{\theta,2}(\mathbb{X})& \nabla_{r} \mathbb{F}_{\theta,2}(\mathbb{X}) \\
	(x^{-})^{\text{T}}&(z^{-})^{\text{T}}& 2r+\varepsilon  \\
	\end{pmatrix},
	\end{equation*}
	Since $ \mathbf{X}^{*}=(x^{*},z^{*}) $ is a solution of LCP, we have
	\begin{enumerate}
		\item The derivative of $ \mathbb{F}_{\theta,~2}(\mathbf{X},r) $ with respect to $ x$ is:  
		\begin{equation*}
		\nabla_{x} \mathbb{F}_{\theta,2} (x^{*},z^{*},r)=\text{diag}\left(\left(\dfrac{r}{x^{*}+r}\right)^{2}\right)_{n \times n}, 
		\end{equation*}
		when $ r $  goes to $ 0 $ and in view of the strict complementary we have to consider two cases:
		\begin{itemize}
			\item  $ x_{i}^{*} \to 0, $ and $ z_{i}^{*}>0~~ $  $\forall i \in \{1, ..., n\}$ then \\
			\begin{equation*}
			\begin{split}
			\lim\limits_{\substack{r \rightarrow 0 \\ x_{i}^{*}\rightarrow 0}}\nabla_{x} \mathbb{F}_{\theta,2} (x_{i}^{*},z_{i}^{*},r)=&  \lim\limits_{\substack{r \rightarrow 0}} \left(\dfrac{r}{\smallO{r}+r}\right)^{2}\\
			=& \lim\limits_{\substack{r \rightarrow 0 }} \left(\dfrac{r}{r}\right)^{2}=1.
			\end{split}
			\end{equation*}
		\end{itemize} 
		\begin{itemize}
			\item  $ x_{i}^{*}>0, $ and $ z_{i}^{*} \to 0~~ $ $\forall i \in \{1, ..., n\}$ then 
		\end{itemize}
		\begin{equation*}
		\lim\limits_{\substack{r \rightarrow 0 \\ z_{i}^{*}\rightarrow 0}}\nabla_{x} 	\mathbb{F}_{\theta,2} (x_{i}^{*},z_{i}^{*},r)=\lim\limits_{\substack{r \rightarrow 0\\ z_{i}^{*}\rightarrow 0}} \left(\dfrac{r}{x_{i}^{*}+r}\right)^{2}=0.
		\end{equation*}
		\item The derivative of $ \mathbb{F}_{\theta,~2}(\mathbf{X},r) $ with respect to $ z $ is: 
		\begin{equation*}
		\nabla_{z} \mathbb{F}_{\theta,2} (x^{*},z^{*},r)=\text{diag}\left(\left(\dfrac{r}{z^{*}+r}\right)^{2}\right)_{n \times n},   
		\end{equation*}
		when $ r $  goes to $ 0 $ and in view of the strict complementary we have to consider two cases:
		\begin{itemize}
			\item  $ x_{i}^{*} \to 0, $ and $ z_{i}^{*}>0~~ $ $\forall i \in \{1, ..., n\}$ then 
		\end{itemize}
		\begin{equation*}
		\lim\limits_{\substack{r \rightarrow 0 \\ x_{i}^{*}\rightarrow 0}}\nabla_{z} 	\mathbb{F}_{\theta,2} (x_{i}^{*},z_{i}^{*},r)=\lim\limits_{\substack{r \rightarrow 0\\ x_{i}^{*}\rightarrow 0}} \left(\dfrac{r}{z_{i}^{*}+r}\right)^{2}=0.
		\end{equation*}
		\begin{itemize}
			\item  $ x_{i}^{*}>0, $ and $ z_{i}^{*} \to 0~~ $ $\forall i \in \{1, ..., n\}$ then 
		\end{itemize}
		\begin{equation*}
		\lim\limits_{\substack{r \rightarrow 0 \\ z_{i}^{*}\rightarrow 0}}\nabla_{z} 	\mathbb{F}_{\theta,2} (x_{i}^{*},z_{i}^{*},r)=\lim\limits_{\substack{r \rightarrow 0}} \left(\dfrac{r}{\smallO{r}+r}\right)^{2}=1.
		\end{equation*}
		\item The derivative of $ \mathbb{F}_{\theta,~2}(\mathbf{X},r) $ with respect to $ r$ is:  
		\begin{equation*}
		\nabla_{r} \mathbb{F}_{\theta,2} (x^{*},z^{*},r)=\left(\left(\frac{x^{*}}{x^{*}+r}\right)^{2}+ \left(\frac{z^{*}}{z^{*}+r}\right)^{2}-\mathbf{1}\right)_{n \times 1},   
		\end{equation*}
		when $ r $  goes to $ 0 $ then we have to prove that the vector $ \dfrac{x^{*2}}{(x^{*}+r)^{2}}+\dfrac{z^{*2}}{(z^{*}+r)^{2}}-\mathbf{1} $  is bounded, since $ 0 \leq \left(\frac{x_{i}^{*}}{x_{i}^{*}+r}\right)^{2} \leq 1 $ and $ 0 \leq \left(\frac{z_{i}^{*}}{z_{i}^{*}+r}\right)^{2} \leq 1~~ $~~ $\forall i \in \{1, ..., n\}$  then we have 
		\begin{equation*}
		-1 \leq \left(\frac{x_{i}^{*}}{x_{i}^{*}+r}\right)^{2}+ \left(\frac{z_{i}^{*}}{z_{i}^{*}+r}\right)^{2}-1  \leq 1, ~~~ \forall i \in \{1, ..., n\}.
		\end{equation*} 
	\end{enumerate}
	Finally, since $ \mathbf{X}^{*}=(x^{*},z^{*}) $ is a solution of LCP, we have $ x^{*}\geq 0 $ and  $ z^{*}\geq 0 $, so that $ x^{-}=z^{-}=0. $
	Hence 
	\begin{equation*}
	\lim_{r \to 0}\nabla_{\mathbb{X}}	\mathbb{F}_{\theta}(\mathbf{X}^{*},r) =\begin{bmatrix}
	\begin{pmatrix}
	M_{n\times n} & -I_{n \times n}  \\
	\phi(Z^{*}) & \phi(X^{*}) \\
	\end{pmatrix}
	&        
	\begin{matrix}
	0 \\[3mm]
	0 \\[3mm]
	\end{matrix}
	\\ 
	0\quad  \quad  \quad 0      &      \varepsilon  \\
	\end{bmatrix},
	\end{equation*}
	and 
	\begin{equation*}
	\lim_{r \to 0} \left\vert\nabla_{\mathbb{X}}	\mathbb{F}_{\theta}(\mathbf{X}^{*},r)\right 
	\vert   
	=\varepsilon \left\vert
	\left(\begin{array}{cc}
	M&-I\\
	\phi(Z^{*})& \phi(X^{*})
	\end{array}\right) \right 
	\vert,
	\end{equation*}
	in view of Lemma \eqref{22786633}, we conclude that if  $ \nabla_{\mathbf{X}}F_{0}(\mathbf{X}^{*})~~ $is invertible  then $ \lim\limits_{\substack{r \rightarrow 0}}\nabla_{\mathbb{X}}	\mathbb{F}_{\theta}(\mathbf{X}^{*},r)  $ is invertible. Hence the two Jacobian matrices are singular or nonsingular at the same time.
\end{proof}
From now on, the enlarged equation $\eqref{889}$ and  $\eqref{88}$   are selected as the reference system in the design of our new algorithms. The idea is simply to apply the standard Newton method to the smooth system $\eqref{889}$ and $\eqref{88}.$ To enforce a global convergence behavior, we also recommend using Armijo's line search. By Lemma \eqref{19624}, we assume that M has  all its principal minors are nonzeros to ensure the convergence of the two algorithms below.\\ 
Now, we present  the new algorithm for our methods described above:
\vspace{0.2cm}\\
\begin{tabular}{llll}
	\hline 
	$ \mathbf{Algorithm~1}$ Nonparametric TLCP with Armijo line search \\
	\hline
	1.~ Chose  $\mathbb{X}^{0} =(\mathbf{X}^{0},r^{0}), ~\mathbf{X}^{0}>0,~r^{0}=<x^{0},z^{0}>/n,~\tau \in (1,1/2), \varsigma \in(0,1).~ $ Set  $k=0. $ \\
	2.~ If $ \mathbb{F}_{\theta}(\mathbb{X}^{k})=0,~$stop.\\
	3.~ Find a direction $ d^{k} \in \R^{2n+1}$ such that \\
	~~~~~~~~~\quad~~~\quad\quad\quad\quad\quad\quad\quad\quad~~~~~~~~~~~~\quad~~~~~~~~~~~~~~~~~~~~~~~~~~~~~~~~~~~~~~~~~~~~~~~~~~~~~$~~~~~~~~~~~~~~~~~~~~~~~\quad~~~~~~~~~~\mathbb{F}_{\theta}(\mathbb{X}^{k})+\nabla_{\mathbb{X}}\mathbb{F}_{\theta}(\mathbb{X}^{k})d^{k}=0.  $\\
	4.~ Choose $ \alpha^{k}=\varsigma^{j_{k}}\in (0,1),~ $ where  $ j_{k} \in \N $ is the smallest integer such that \\
	~~~~~~~~~\quad~~~\quad\quad\quad\quad\quad\quad\quad\quad~~~~~~~~~~~~\quad~~~~~~~~~~~~~~~~~~~~~~~~~~~~~~~~~~~~~~~~~~~~~~~~~~~~~
	$~~~~~~~~~~~~~~~~~~~~~~~~~~~~~~~~~~~~~~~~\Theta_{\theta}(\mathbb{X}^{k}+\varsigma^{j_{k}}d^{k})\leq (1-2\tau\varsigma^{j_{k}})~\Theta_{\theta}(\mathbb{X}^{k}).  $\\
	5. ~Set $ \mathbb{X}^{k+1}=\mathbb{X}^{k}+\alpha^{k}d^{k}~ $ and $ k\gets k+1.~ $Go to step $ 2.$ \\
	\hline
	
\end{tabular}
\\
\\

\begin{tabular}{llll}
	\hline
	$  \mathbf{Algorithm~2}$ Nonparametric Soft-LCP method with Armijo line search \\
	\hline
	1.~ Chose  $\mathbb{X}^{0} =(\mathbf{X}^{0},r^{0}), ~\mathbf{X}^{0}>0,~r^{0}=<x^{0},z^{0}>/n,~\tau \in (1,1/2), \varsigma \in(0,1).~ $ Set  $k=0. $ \\
	2.~ If $ \mathbb{F}_{s}(\mathbb{X}^{k})=0,~$stop.\\
	3.~ Find a direction $ d^{k} \in \R^{2n+1}$ such that \\
	~~~~~~~~~\quad~~~\quad\quad\quad\quad\quad\quad\quad\quad~~~~~~~~~~~~\quad~~~~~~~~~~~~~~~~~~~~~~~~~~~~~~~~~~~~~~~~~~~~~~~~~~~~~$~~~~~~~~~~~~~~~~~~~~~~~~~~~~~~~~~~~~~~~~\mathbb{F}_{s}(\mathbb{X}^{k})+\nabla_{\mathbb{X}}\mathbb{F}_{s}(\mathbb{X}^{k})d^{k}=0.  $\\
	4.~ Choose $ \alpha^{k}=\varsigma^{j_{k}}\in (0,1),~ $ where  $ j_{k} \in \N $ is the smallest integer such that \\
	~~~~~~~~~\quad~~~\quad\quad\quad\quad\quad\quad\quad\quad~~~~~~~~~~~~\quad~~~~~~~~~~~~~~~~~~~~~~~~~~~~~~~~~~~~~~~~~~~~~~~~~~~~~
	$~~~~~~~~~~~~~~~~~~~~~~~~~~~~~~~~~~~~~~~~\Theta_{s}(\mathbb{X}^{k}+\varsigma^{j_{k}}d^{k})\leq (1-2\tau\varsigma^{j_{k}})~\Theta_{s}(\mathbb{X}^{k}).  $\\
	5. ~Set $ \mathbb{X}^{k+1}=\mathbb{X}^{k}+\alpha^{k}d^{k}~ $ and $ k\gets k+1.~ $Go to step $ 2.$ \\
	\hline	
\end{tabular}
\\

where the merit function used in the line search is  
\begin{equation*}
\Theta_{\theta}(\mathbb{X})=\frac{1}{2}\Vert \mathbb{F}_{\theta}(\mathbb{X}) \Vert^{2}~~  (resp. ~~ \Theta_{s}(\mathbb{X})=\frac{1}{2}\Vert \mathbb{F}_{s}(\mathbb{X}) \Vert^{2}).
\end{equation*}
\section{Numerical Results}
Through this article, we studied two methods Soft-LCP and TLCP to solve the LCP, we present in this section some numerical experiments. First, we present a comparison on some randomly generated problems of our two methods with other approaches that have been suggested recently in \cite{El Ghami, Tran}. \\Then, we study two concrete examples, the first one is a second order ordinary differential equation and the second is an obstacle problem that can be formulated as LCP \eqref{1}.\\
Finally We tested our algorithms on several absolute value equations problems. Our results are very promising and outperform standard methods. \\ 
For all the numerical tests and all the considered methods, the used codes are simple Matlab codes. We restrict our choice of $ \theta$-function to  $ \theta_{r}^{1}(x)=\frac{x}{x+r} $. \\
Our aim is to validate our approach and run some preliminary comparison with other methods, and not to optimize the performance of the algorithm. 
\subsection{Comparisons of methods for LCPs}	
\vspace{0.1cm}
We generate for several problem sizes, n=32,~64,~128,~256 the data $ (M,q) $ in order to have a solution for LCP as follows
\begin{center}
	\colorbox[rgb]{0.74,0.98,1}{
		\begin{minipage}{0.98\textwidth}
			R=\text{rand}(n,~n);\\
			M=R$^{'}$*R+n*\text{eye}(n);\\
			h=\text{rand}(n) ;\\
			z=\text{round}(h).*\text{rand}(n,~1);\\
			t=(1-\text{round}(h)).*\text{rand}(n,~1);\\
			q=-M*t=z.\\ 
	\end{minipage}}
\end{center}
We compare our two methods denoted Soft-LCP and TLCP with other methods:
\begin{itemize}
	\item TLCP2 method which is the same algorithm with a different formulation for the complementarity 
	\begin{equation*}
	\theta_{r}(x_{i})+\theta_{r}(z_{i})-\theta_{r}(x_{i}+z_{i})=0.
	\end{equation*}
	In this case we don't necessarily need the constraint
	\begin{equation*}
	\frac{1}{2} \Vert x^{-}\Vert^{2}+ \frac{1}{2} \Vert z^{-}\Vert^{2}+r^{2}+\varepsilon r=0.
	\end{equation*} 
	since it is a reformulation of the complementarity and not a relaxation (we can use a fixed $ r $).
	\item The classical interior-point method IPM \cite{El Ghami}.
	\item Nonparametric interior-point method NPIPM developped in \cite{Tran}.
\end{itemize}
The main idea of all these methods is to regularize the complementarity condition $ x^{T}z=0$ and solve a system of equations using Newton's method. We use an Infeasible IPM to compare with our methods. Regarding the NPIPM, the technique proposed avoids any parameter management. For TLCP2, we have fixed $ r $ to $ 1 $. We take for all this methods the initial point $ (x_{0},~z_{0})=1,~$ where $ 1 \in \R^{n} $ is the vector whose components are all equal to $ 1 $ and $ r_{0}=\langle x_{0},~z_{0} \rangle/n $ and the precision is set as $ 10^{-6}.$ \\
The comparative results are given in the Table $ 1 $ to $ 5.$ We are interested in the following aspects: the comp.err, computed as $   \vert x^{T}z  \vert$, feas.err computed as $\Vert Mx+q-z \Vert $ the number of iterations nb-iter and the time.
\begin{table}[H]
	\caption{Results from IPM with n=$32, 64, 128, 256.  $ }
	\begin{center}
		\begin{tabular}{cccccc}
			\hline 
			n & comp.err &  feas.err&  r & nb-iter& time \\
			\hline 
			32  & 9.e-7 & 0  & 0  &  198     & 0.5606 \\
			
			64  & 9.e-7  & 1.e-12  & 0  & 212   &  0.9313  \\
			
			128 & 9.e-7 & 4.e-10 & 0  &  248  & 3.6056  \\
			
			256 &  9.e-7 & 0  &  0  &  238  &8.7353  \\
			\hline 
		\end{tabular}
		\label{tab1}
	\end{center}
\end{table}
\begin{table}[H]
	\caption{Results from NPIPM with n=$32, 64, 128, 256.  $ }
	\begin{center}

		\begin{tabular}{cccccc}
			\hline  n & comp.err  &  feas.err&  r & nb-iter& time \\
			\hline
			
			32  & 5.e-7  & 0 & 0  & 22 &0.0541  \\
			64  & 1.5e-5 & 0 & 0  &  165  &   1.4925 \\
			128 & 2.e-7 & 0   & 0 &   230  &2.2394  \\
			256 & 1.2e-8  & 0  & 0  &  281  & 12.8129  \\
			\hline 
		\end{tabular}
		\label{tab2}
	\end{center}	
\end{table}

\begin{table}[H]
	\caption{Results from TLCP with n=$32, 64, 128, 256.  $ }
	\begin{center}
		\begin{tabular}{cccccc}
			\hline  n & comp.err  &  feas.err&  r & nb-iter& time \\
			\hline 
			32  & 2.e-7  & 1.e-13 & 1.e-3 & 10    & 0.1486  \\
			64  & 8.1e-8  & 1.e-12 & 1.e-3 &  10  &  0.0342  \\
			128 & 2.e-7 & 0   & 5.e-4&  11   & 0.0995 \\
			256 & 1.e-7 & 0  &  2.e-4  & 12  & 0.4620  \\
			\hline 
		\end{tabular}
		\label{tab3}
	\end{center}
\end{table}
\begin{table}[H]
	\caption{Results from TLCP2 with n=$32, 64, 128, 256.  $ }
	\begin{center}
		\begin{tabular}{cccccc}
			\hline  n & comp.err &  feas.err&  r & nb-iter& time\\
			\hline  
			32 & 0 & 8.e-9  & 1 & 11    & 0.0154  \\
			64  & 0 & 2.e-9 & 1  &  12  &  0.0429  \\
			128 & 2.e-7 & 1.6e-8  & 1 &  11   & 0.1993  \\
			256 & 8.e-8  & 8.6e-7 &  1  &  38  & 1.6799 \\
			\hline 
		\end{tabular}
		\label{tab4}
	\end{center}
\end{table}
\begin{table}[H]
	\caption{Results from Soft-LCP with n=$32, 64, 128, 256.  $ }
	\begin{center}
		\begin{tabular}{cccccc}
			\hline  n & errcomp &  erfeas&  r & nb-iter& time \\
			\hline 
			32  & 1.e-5 & 0  & 8.e-3 & 14    & 0.0233  \\
			64  & 1.e-6  & 1.e-13 & 5.2e-3  &  18  & 0.0607  \\
			$  128$ & 9.e-7 & 0   & 5.e-3&  20   & 0.3680 \\
			256 & 1.e-6 & 0  &  3.9e-3  &  22 & 0.9068  \\
			\hline 
		\end{tabular}
		\label{tab5}
	\end{center}
\end{table}
In the above comparisons, we notice that our methods have much better results in terms of iteration numbers and CPU-time than classic interior-point-method IPM and NPIPM. The TLCP method requires the fewest iteration numbers.
\subsection{An obstacle problem} \label{98456}
Let $f$ and $g$ two continuous functions defined in  $[0,1]$. We want to solve the following obstacle problem: \\find
$u:[0, 1] \rightarrow \R$ such that:
\begin{equation*}
\left \{ \begin{array}{rcc}
-u^{"}(x) &\geq& f(x)\\
u(x) &\geq& g(x)\\
(-u^{"}(x) -f(x))(u(x) -g(x))&=&0
\end{array}\right. \textrm{on}\, \,]0,1[, 
\end{equation*}
and $ u(0)=u(1)=0.$\\
The first equation means a maximum concavity of the function u. In the second equation, we want the solution $u$ to be above $g$. In the third equation, we have at least equality in one of the two previous equations.
In order to get a linear complementarity problem, we set $z=u-g$ and we discretize by using the finite difference. We introduce a uniform subdivision $x_i = i*h, i=0,\dots N+1$
of $[0, 1]$, where $h =\frac{1}{N + 1}$.\\
We use the second-order centered finite difference to approximate the second order derivatives $ z^{''}(x) $ and $ g^{''}(x).$ We then try to solve the following problem:
\begin{equation*}
\left\{ \begin{array}{rcc}
\frac{-z_{i-1}+2z_{i}-z_{i+1}}{h^{2}}+\frac{-g_{i-1}+2g_{i}-g_{i+1}}{h^{2}}-f_{i}&\geq& 0\\
z_{i} &\geq& 0\\
\left(\frac{-z_{i-1}+2z_{i}+z_{i+1}}{h^{2}}+\frac{-g_{i-1}+2g_{i}-g_{i+1}}{h^{2}} -f_{i}\right)(z_{i})&=&0
\end{array}\right. ,
\end{equation*}
$ \textrm{for}\,i=1...N, u_{0}=u_{N+1}=0. $\\
Where $ g_{i}=g(x_{i}),~f_{i}=f(x_{i}),~z_{i}=z(x_{i})$ and $ u_{i}=u(x_{i}).$ We obtain the following complementarity problem:
\begin{equation*}
\begin{array}{rcc}
(Mz+q)^Tz&=& 0\\
z &\geq& 0\\
Mz+q&\geq&0
\end{array}
\end{equation*}
where
\begin{equation*}
M=\frac{1}{h^2}\left(\begin{array}{cccc}
2&-1& &\\
-1& \ddots & \ddots &\\
& \ddots& \ddots & -1\\
&& -1&2\\
\end{array}\right)
\end{equation*}
and $q=Mg-f$. If $ 1 $ is not an eigenvalue of $ M $ is equivalent to AVE, (\cite{O. L. Mangasarian2}, Prop. 2),
\begin{equation*}
(M-I)^{-1}(M+I)x-\vert x\vert=(M-I)^{-1}q.  
\end{equation*}
We present in the following figures, the results of our two methods and LPM method from \cite{O. L. Mangasarian2}. The obstacle $ g $ is chosen here to be \\ $g(x)=\text{max}(0.8-20*(x-0.2)^2,\text{max}(1-20(x-0.75)^2,1.2-30(x-0.41)^2)),~ f(x)=1$ and $N=50.$

\begin{figure}[H]\label{fig1}
	\begin{minipage}[H]{.46\linewidth}
		\begin{center}
			\includegraphics[width=9cm,height=7cm]{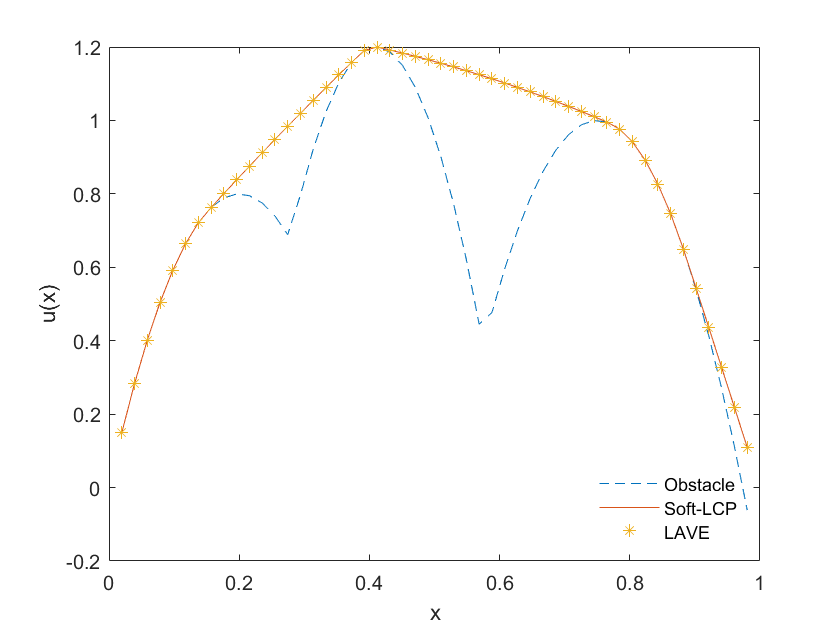}
		\end{center}
	\end{minipage} \hfill
	\begin{minipage}[H]{.46\linewidth}
		\begin{center}
			\includegraphics[width=9cm,height=7cm]{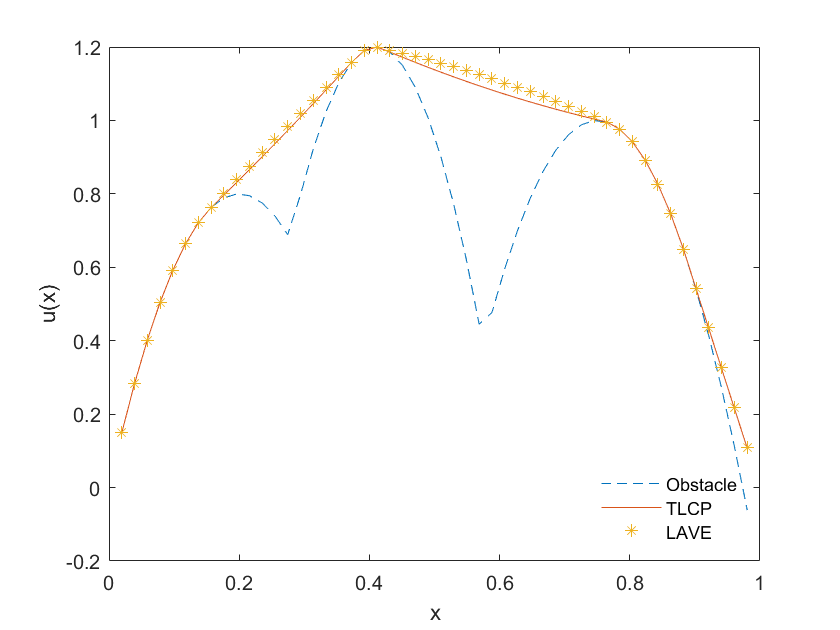}
		\end{center}
	\end{minipage}
	\caption{Numerical solution of the obstacle problem \eqref{98456} with TCLP, Soft-LCP methods, and method from \cite{O. L. Mangasarian2}.}
\end{figure}
We remark that the both TLCP, Soft-LCP, and LPM method \cite{O. L. Mangasarian2} have 19 common points on the curve g and none below g over 50 points. This example also confirms that our approach, TLCP and Soft-LCP  method gives consistent results.
\subsection{An ordinary differential equation} \label{234579800}
We consider the ordinary differential equation 	
\begin{equation}\label{2233}
x^{''}(t)-|x(t)|=-2-t,~~~x(0)=-1, ~~~~x^{'}(0)=1, ~~~ t\in [0,~5].
\end{equation}
First, we discretize the EDO equation by using the finite difference scheme. We use the second-order centred finite difference to approximate the second order derivative
\begin{equation}\label{3333}
\dfrac{x_{i-2}-2x_{i-1}+x_{i}}{h^{2}}-|x_{i}|=(-2-t)_{i}.
\end{equation}
Equation $\eqref{3333}$ was derived with equispace gridpoints $ t_{i}=ih, ~i=1, ...N.$ In order to approximate the Neumann boundary conditions we use a center difference 
\begin{equation}\label{333}
\dfrac{x_{1}-x_{-1}}{2h}=x^{'}(0)=1.
\end{equation}
Using the classical decomposition of the absolute value \cite{Lina Abdallah} we reformulate $\eqref{3333}$ as follows 
\begin{equation}\label{33344}
\left\{
\begin{array}{llllll} 
N_{1}x^{+}-N_{2}x^{-}=q,
~~~~~~~~~~~~~~~~~~~~~~~~~~~~~~~~~~~~~~~~~~~~  \\
0 \leq x^{+} \perp x^{-} \geq 0,
\end{array}
\right.
\end{equation}

where 
\begin{equation*} 
N_{1}=\frac{1}{h^2}\left(\begin{array}{ccccc}
2-h^{2}&         &        &   &     \\
-2  &1-h^{2}  &       &   &    \\
1      & \ddots  & \ddots &   &    \\
&\ddots   &  \ddots      & \ddots  &   \\
&   &1       & -2  &  1-h^{2}\\
\end{array}\right), ~~ N_{2}=\frac{1}{h^2}\left(\begin{array}{ccccc}
2+h^{2}&         &        &   &     \\
-2  &1+h^{2}  &       &   &    \\
1      & \ddots  & \ddots &   &    \\
&\ddots   &  \ddots      & \ddots  &   \\
&   &1       & -2  &  1+h^{2}\\
\end{array}\right),	
\end{equation*}
and 
$  
q=-\frac{1}{h^2}\left(\begin{array}{c}
2-2h\\
-1     \\
\vdots        \\
0\\
\end{array}\right)
-\left(\begin{array}{c}
2+h\\
2+2h \\
\vdots    \\
2+Nh\\
\end{array}\right)
$. \\

\vspace{0.2cm}
$ N_{1} $ is invertible, then the problem $\eqref{33344}$  is reduced to a standard LCP.\\
We compare the obtained solution by Soft-LCP and TLCP to the predefined Runge-Kutta ode45 function in MATLAB \cite{Matlab}. The domain is $ t \in [0, 5] $, initial conditions $ x({0})=-1, x^{'}(0)=1$ and $ N=100.$
\begin{figure}[H]
	\begin{minipage}[H]{.46\linewidth}
		\begin{center}
			\includegraphics[width=9cm,height=7cm]{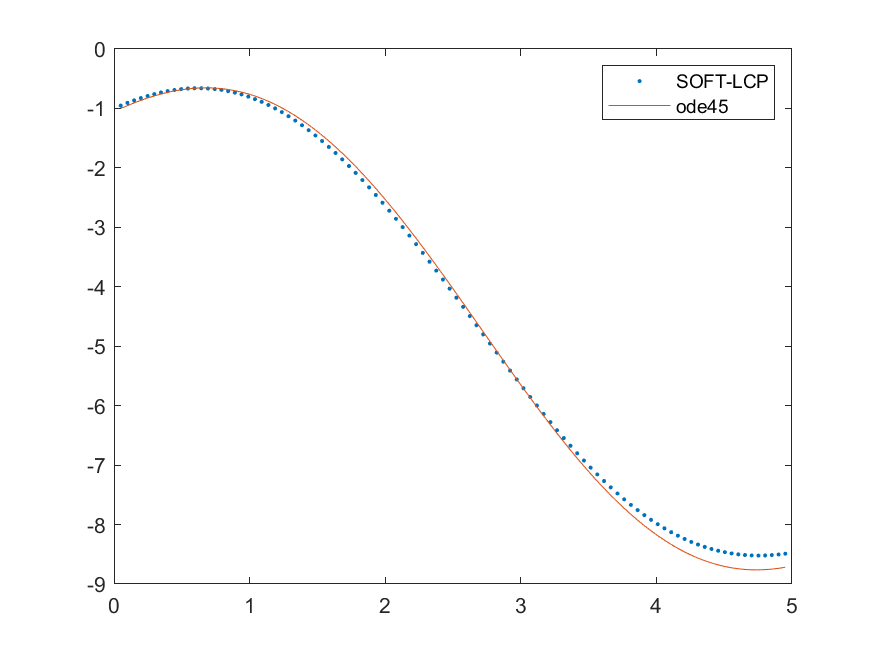}
		\end{center}
	\end{minipage} \hfill
	\begin{minipage}[H]{.46\linewidth}
		\begin{center}
			\includegraphics[width=9cm,height=7cm]{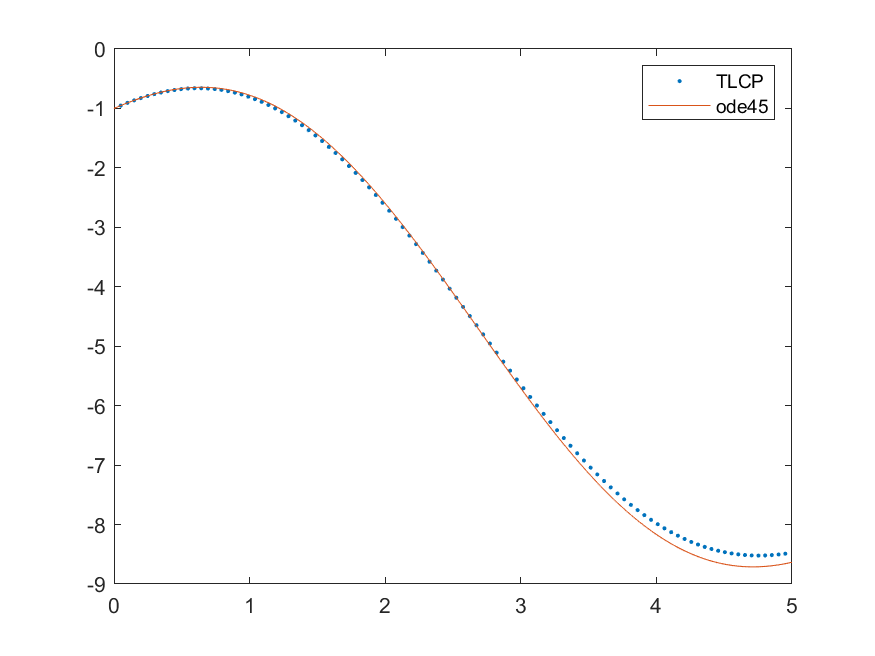}
		\end{center}
	\end{minipage}
	\caption{Numerical solution of  $\eqref{234579800}$ with ode45 and both methods}
\end{figure}
Both methods solve the problem and gives consistent results.
\subsection{Application to Absolute Value Equation }
We consider the absolute value equation AVE, defined as 
\begin{equation} \label{1993}
Ax-\vert x\vert=b,
\end{equation}
with $ A\in \R^{n \times n} $ and $ b\in \R^{n}.$ We studied two  cases where AVE has a unique solution and for general AVE. Using the same technique as in \cite{Lina Abdallah}, \eqref{1993} can be cast as the following complementarity problem 
\begin{equation} 
A(x^{+}-x^{-})-(x^{+}+x^{-})=b,~~~~\quad 0\leq x^{+}\perp x^{-} \geq 0., 
\end{equation}
equivalent to 
\begin{equation} 
(A-I)x^{+}=(A+I)x^{-}+b,~ \quad 0\leq x^{+}\perp x^{-} \geq 0, 
\end{equation}
where $ x^{+}=\text{max}(x,0) $  and  $ x^{-}=\text{max}(-x,0) $. This decompsition guarantes that $|x|=x^{+}+x^{-}.$ So AVE can be cast as the following LCP 
\begin{equation} 
x^{+}=M	x^{-}+q,~ 0\leq x^{+}\perp x^{-} \geq 0, 
\end{equation}
with $ M=(A-I)^{-1}(A+I)$ and $ q=(A-I)^{-1}b. $ 
\subsubsection{Random uniquely solvable generated problem}
We consider the special case where AVE is uniquely solvable, to guarantee the convergence of the Newton method. One way to generate such (AVE) is to generate a matrix $ A $ with singular values exceeding $ 1.$ We first chose a random $ A $ from a uniform distribution on $ [-10, 10], $ then we chose a random $ x $ from a uniform distribution on $ [-1, 1]$. Finally we computed $ b=Ax-\vert x\vert $ . We ensured that the singular values of each $ A $ exceeded $ 1 $ by actually computing the minimum singular value and rescaling $ A $ by dividing it by the minimum singular value multiplied by a random number in the interval $ [0, 1] $. We generate for the several values for $ n=32,~64,~128,~256,~512,~1024$, the data $ (A,~b)$ by the following Matlab code in order to have a solution for AVE:
\begin{center}
	\colorbox[rgb]{0.74,0.98,1}{
		\begin{minipage}{0.98\textwidth}
			n=input('dimension of matrix A=');\\
			R=10*(rand(n,n)-rand(n,n));\\
			A=R/(min(svd(R))*rand(1));\\
			x=rand(n,1)-rand(n,1);\\
			b=A*x-abs(x). 
	\end{minipage}}
\end{center}
The required precision for solving AVE is $ 10^{-6}$. For each $ n$ we consider $ 100$ instances. \\
Now, we compare our methods Soft-LCP and TLCP to Generalized Newton method from \cite{Olvi L. Mangasarian}, which is denoted GN. In this method, we solve each iteration a linear system: 
\begin{equation} 
(A-D(x^{i}))x^{i+1}=b 
\end{equation}
where $ D(x^{i})=\text{diag}(\text{sign}(x^{i})) $.
Results are summarized in Table 6, which gives the number of iterations, the time required to solve all the $ 100$ instances. Our methods solve all $ 100 $ AVEs to an accuracy of $ 10^{-6}$ and validate our approach. We notice that the GN method is the fastest because at each iteration it solves only one linear system, the TLCP method gives the fewest iterations to solve the 100 instances.
\begin{table}[H]
	\caption{Comparison of Soft-LCP and TLCP with GN method, in the case with singular values of $ A$ exceeds $ 1 $ for $ 100 $ randomly generated AVE  of size $ n. $ }
	\begin{center}
		\begin{tabular}{ccccccc}
			\hline  $ n $ & it-Soft-LCP &Time-Soft-LCP&  it-TLCP & Time-TLCP& it-GN& Time-GN \\
			\hline 
			\vspace{0.1cm} 
			32  & 201  & 0.0394  & 104  &  0.0238 & 255 & 0.0071 \\
			64 & 201 & 0.1041   & 107&  0.0646   & 274&  0.0182 \\
			128 & 200 & 0.2844  & 111  &  0.1767 & 274 &  0.0641 \\
			256 & 212 & 1.6986  & 106 &  0.9727  & 290 &  0.2301 \\
			512 & 284 & 11.0947 &  110  &  5.0497  & 295 &  1.2925 \\
			1024 & 284 & 42.1565 &  111  &  45.3930 & 291 &  14.8541 \\
			\hline 
		\end{tabular}
	\end{center}
\end{table}
\subsubsection{Random generated problem}
Now we present results for general AVE, which is the main interest of our algorithm. The data are generated like \cite{O. L. Mangasarian2} for several $ n $ and for several values of the parameteres, in each situation we solve $ 100 $ instances of the problem. We choose a random $ A $ from a uniform distributin on $ [-10, 10] $, then chose a random $ x $ from a uniform distribution on $ [-1, 1] $ and set $ b=Ax-\vert x \vert.$ The data $ (A, b) $ are generated by Matlab script: 
\begin{center}
	\colorbox[rgb]{0.74,0.98,1}{
		\begin{minipage}{0.98\textwidth}
			n=input('dimension of matrix A=');\\
			rand('state',0);\\
			A=10*(rand(n,n)-rand(n,n));\\
			x=rand(n,1)-rand(n,1);\\
			b=A*x-abs(x);
	\end{minipage}}
\end{center}
We will compare 4 methods valid for general AVE:
\begin{itemize}
	\item TLCP method from Algorithm 1;
	\item Soft-LCP method from Algorithm 2;
	\item Concave minimization method CMM from \cite{O. L. Mangasarian2};
	\item Successive linear programming method LPM from \cite{O.L. Mangasarian1};
\end{itemize}

In table 7-10, "nnztot" gives the number of violated expressions for all problems, "nnzx" gives the maximum violated expressions for one problem, "nb-iter" gives the number of iteration for all the problems. We also provide the time in seconds and the number of problems where we did not manage to solve AVE.
\begin{table}[H]
	\caption{Results from TLCP on  with $ 100$ consecutive random AVEs}
	\begin{center}
		\begin{tabular}{cccccc}
			\hline 
			n &nnztot&nnzx &nb-iter &  time& nb-failure\\
			\hline 
			\vspace{0.1cm}
			32  & 3& 1& 1647   & 0.6274  &  3  \\
			
			64  & 5& 1& 1776  & 1.2548  &  5   \\
			
			128 & 5& 1& 2359 & 2.4182    &   7  \\
			
			256 & 8&1 &  2448 & 22.8817  & 8  \\
			\hline 
		\end{tabular}
	\end{center}
\end{table}
\begin{table}[H]
	\caption{Results from Soft-LCP on  with $ 100$ consecutive random AVEs}
	\begin{center}
		\begin{tabular}{cccccc}
			\hline 
			n & nnztot& nnzx&nb-iter &  time&nb-failure\\
			\hline 
			\vspace{0.1cm}
			32  &1 &1 & 960   & 0.4287  & 1  \\
			
			64  & 1&1 & 1032  & 0.8351  &  1    \\
			
			128 &3 & 1& 1478 & 1.6692    &  3   \\
			
			256 &1 &1 &  1996 & 18.3965 &   1   \\
			\hline 
		\end{tabular}
	\end{center}
\end{table}
\begin{table}[H]
	\caption{Results from CMM on  with $ 100$ consecutive random AVEs}
	\begin{center}
		\begin{tabular}{cccccc}
			\hline 
			n &nnztot&nnzx& nb-iter &  time &nb-failure\\
			\hline 
			\vspace{0.1cm}
			32  & 13&1 &640    & 4.2832   &  13 \\
			
			64  &11 &1 & 588  & 7.0034  &   11   \\
			
			128 &13& 1& 693 & 19.9940    &  13   \\
			
			256 &15& 1&  753& 143.6931 &  15  \\
			\hline 
		\end{tabular}
	\end{center}
\end{table}
\begin{table}[H]
	\caption{Results from LPM on  with $ 100$ consecutive random AVEs}
	\begin{center}
		\begin{tabular}{cccccc}
			\hline 
			n & nnztot&nnzx& nb-iter &  time&   nb-failure\\
			\hline 
			\vspace{0.1cm}
			32  &8&1 & 313  & 2.2422 &  8  \\
			
			64  &19&4& 411  & 6.0978  &   18  \\
			
			128 &21&3& 433 & 18.2642    &  20   \\
			
			256 &29&5&  606 & 156.4612  &  22   \\
			\hline 
		\end{tabular}
	\end{center}
\end{table}
In every cases our methods manage to reduce the number of unsolved problem, which was our principal aim. In every case it gives the smallest number of unsolved problem in a very reasonable time.
\section{Conclusion}
In this paper, we propose two methods to solve the LCP. A complete analysis is provided to validate our approach.
Furthermore, a numerical study shows that our approach is interesting. Numerical experiments on several LCP problems and a comparison with some existing methods proves the efficiency of our study.\\
We have presented an application of absolute value equation AVE and two examples (an obstacle problem and ODE) and show that our two methods are promising.

\end{document}